\documentclass[ reqno]{amsart}
\usepackage{graphicx} 

\title[Geometric Aspects of $C^*$-Extreme Points]{Geometric Aspects of $C^*$-Extreme Points}
 \author[N. Hotwani]{Neha Hotwani${}^1$}
\address{Department of Mathematics\\
Shiv Nadar Institution of Eminence. Gautam Buddha
Nagar-201314, India}
\email{nehahotwani19@gmail.com}

\author[T.S.S.R.K. Rao]{T.S.S.R.K. Rao${}^2$}
\address{Department of Mathematics\\
Shiv Nadar Institution of Eminence. Gautam Buddha
Nagar-201314, India}
\email{srin@fulbrightmail.org}


 
\usepackage{amsmath, amsfonts, amssymb, setspace, color, enumerate, bbold}
      \newtheorem{theorem}{Theorem}[section]

      \newtheorem{remark}[theorem]{Remark}
       \newtheorem{corollary}{Corollary}
      \newtheorem{lemma}[theorem]{Lemma}
      
      \newtheorem{proposition}[theorem]{Proposition}

      \def\N{{\mathbb N}}

      \def\cN{\mathcal C}

      \def\cI{\mathcal I}

      \def\cM{\mathcal M}
       \def\cN{\mathcal N}
       
      \def\cR{\mathcal R}

      \def\cA{\mathcal A}
      \def\cB{\mathcal B}
      
      \def\cC{\mathcal C}

      \def\bb1{\mathbb 1}
      
      \def\b1{\bold{1}}

\usepackage[dvipsnames]{xcolor}
\usepackage[pagebackref,colorlinks,linkcolor=Maroon,citecolor=MidnightBlue,urlcolor=NavyBlue,hypertexnames=true]{hyperref}

\usepackage{ulem}
\usepackage{titlesec}
\usepackage{xstring}

\titleformat{\section}
  {\normalfont\Large\bfseries\centering} 
  {\thesection}{1em}{}                   

\renewcommand{\thesection}{\arabic{section}}

\newcommand{\range}{\operatorname{Range}}
\newcommand{\dist}{\operatorname{dist}}

\usepackage{ulem}
\date{}

\usepackage[dvipsnames]{xcolor}

\newtheorem{letterthm}{Theorem}

\usepackage[pagebackref,colorlinks,linkcolor=Maroon,citecolor=MidnightBlue,urlcolor=NavyBlue,hypertexnames=true]{hyperref}

\numberwithin{equation}{section}

\begin{document}

\begin{abstract}
We provide a characterization of the $C^*$-extreme points of the closed unit ball of a von Neumann algebra and demonstrate that $C^*$-extremality is equivalent to both linear extremality and strong extremality. As an application, we characterize certain classes of von Neumann algebras in terms of their $C^*$-extreme points.
\end{abstract}
\maketitle

\section{Introduction}
\label{sec:Intro}
In the study of operator algebras, the geometry of the unit ball plays a central role in understanding the structure of the underlying algebra. One way to study this geometric structure is by looking at extreme points - elements that cannot be expressed as nontrivial convex combinations of others. See the classical work by \cite{K,AS,P}, where the extremal ideas have been used to classify the structures of $C^*$-algebras, function spaces, etc.

In Banach space theory, for a given Banach space $X$ (that we consider as a subspace of its bidual $X^{**}$ via the canonical embedding),
a linear extreme point of the closed unit ball $X_1$ 
that remain linear extreme in the bidual $X_1^{**}$ is known as \textbf{weak$^{*}$-extreme point}. 
Another well-known, stronger notion of a linear extreme point is that of a strongly extreme point. Recall that $x\in X_1$, is called \textbf{strongly extreme point} if  for any sequences $\{x_n\}$ and $\{y_n\}$ in $X_1$,  $\frac{x_n+y_n}{2} \to x$ implies $x_n-y_n \to 0$. 
For more details, see \cite{DHS}. Note that if $x\in X_1$ is a strongly extreme point, and $x\in Y\subset X$ is a closed subspace, then $x$ is a strongly extreme point of $Y_1$.
It is known that strongly extreme points remain strongly extreme in bidual. It is also known that a weak$^{*}$-extreme point need not be weak$^{*}$-extreme in the bidual, see \cite{SR}. It is interesting to 
compare similar extremal behaviour in the context of 
their non-commutative analogue, called $C^*$-extreme points, in the $C^*$-algebra setup.  As important as classical convexity, the notion of $C^*$-convexity has gathered significant attention recently, for instance, see \cite{HMP, FM,M, Ma, Ma3, Ma4}. 
Recall   that (\cite{LP}) for a unital $C^*$-algebra $\cA$ with identity $\bold{1}$,
 an element $x \in \cA_1$ is said to be a \textbf{$C^*$-convex combination} of $k$ elements $x_1,\dots, x_k\in \cA_1$, if there exist $t_1, \dots, t_k\in \cA$ such that $\sum_{i=1}^kt_i^*t_i=\bold{1}$ and $x=\sum_{i=1}^kt_i^*x_it_i$. The $t_i$'s are known as the \textbf{coefficients} of this $C^*$- convex combination. If the coefficients, i.e., the $t_i$'s are  invertible, then this $C^*$-convex combination is called \textbf{a proper $C^*$- convex combination}. 
 $x\in \cA_1$ is said to be a \textbf{$C^*$-extreme point} of $\cA_1$ if, whenever $x$ can be written as a proper $C^*$-convex combination of $x_1, \dots, x_k \in \cA_1$, that is,
\[
x= \sum_{i=1}^kt_i^*x_it_i,
\]
where $t_1,\dots, t_k\in \cA$ are invertible with $\sum_{i=1}^kt_i^*t_i=\bold{1}$, then each $x_i$ is unitarily equivalent to $x$, i.e., there exist unitaries $u_1, \dots, u_k\in \cA$ such that $x_i=u_i^*xu_i$ for $i=1, \dots, k$.

The closed unit ball $\cA_1$ is $C^*$-convex subset of   $\cA$. For completeness, we include a proof of this well-known fact.

\textbf{Fact:} 
    $\cA_1$ 
   is $C^*$-convex. 
   Let $t_1, \dots, t_k\in \cA$ such that $\sum_{i=1}^kt_i^*t_i=\bold{1}$ and $x_1, \dots, x_k\in \cA_1$. We need to show that $x=\sum_{i=1}^kt_i^*x_it_i\in \cA_1$. Since we know that any $C^*$-algebra can be embedded in $B(H)$, the space of all bounded linear operators on some Hilbert space $H$, it follows that $\cA_1 \subset B(H)_1$. Now, from \cite[Example~3]{LP}, we conclude that $\|x\|\leq 1$. 
    Thus, $\cA_1$ is $C^*$-convex in $\cA$.
   
When ${\mathcal A}$ is an infinite-dimensional $C^\ast$-algebra, we do not know if ${\mathcal A}_1$ is always a $C^\ast$-convex set in ${\mathcal A}^{\ast\ast}$.

Throughout this article, we only consider unital $C^*$-algebras, and we denote such $C^*$-algebras by  $\cA, \cB, \cC$ and von Neumann algebras by $\cM, \cN, \cR$.
The notation $\cA^{**}$ denote the bidual of the $C^*$-algebra $\cA$.
It is well-known that $\mathcal{A}^{**}$ is a $C^*$-algebra (in fact, a von Neumann algebra), and that the canonical embedding $\mathcal{A} \hookrightarrow \mathcal{A}^{**}$ is a $C^*$-algebra homomorphism. 
Similar to our earlier discussion, we are now interested in whether $C^*$-extremality passes from $\cA_1$ to $\cA_1^{**}$.
In this direction, we have succeeded in showing that a linear extreme point of $\cA_1$ is, in fact, a $C^*$-extreme point of $\cA_1^{**}$ (Corollary~~\ref{c:bidual}) . 


This article is organized as follows. In Section \ref{s:ext},
we show that in  a $C^*$-algebra, linear extreme points are weak${^*}$-extreme points. We also show that extremality is preserved under passage to the quotient map via a closed two-sided ideal. 
 Section \ref{sec:C*-ext}
 contains a proof of a Wold decomposition-type theorem in the setting of von Neumann algebras and conclude with the observation that any two similar isometries in a von Neumann algebra are necessarily unitarily equivalent (Theorem~~ \ref{t:siml imp unit equiv}). 

In Section \ref{sec:main result}, we use Theorem \ref{t:siml imp unit equiv} and comparison theorem (\cite[Theorem ~6.2.7]{KR2}) to characterize the $C^*$-extreme points of $\cM_1$. 
One of our main results in this article along
these lines is the following.

\begin{letterthm}
\label{t:A}
Let  $x\in \cM_1$. Then $x$ is a $C^*$-extreme point of $\cM_1$ if and only if 
  there exist central projections $p_1, p_2, p_3 \in \cM$ such that
  $p_1+p_2+p_3=\bold{1}$,
  and the following holds:
  \begin{enumerate} [(i)]
      \item \label{i1:u} either $p_1=0$, or $p_1x$ is a unitary in $p_1\cM$,
      \item \label{i2:iso} either $p_2=0$, or $p_2x$ is a non-unitary  isometry in $p_2\cM$, and,
      \item \label{i3:coiso} either $p_3=0$, or $p_3x$ is a non-unitary coisometry in $p_3\cM$.
  \end{enumerate}
\end{letterthm}

A proof of Theorem \ref{t:A} is given in Section \ref{sec:main result}.
As a consequence of Theorem \ref{t:A}, we characterize certain von Neumann algebras in terms of the  $C^*$-extreme points of $\cM_1$ (Theorem~~\ref{t:C*ext-in-subclass}).
As a further consequence of Theorem \ref{t:A}, 
we establish the following result, another main theorem of this article, which states that the three notions - $C^{*}$-extreme points, linear extreme points, and strongly extreme points - coincide for $\mathcal M_{1}$. Consequently, these points remain $C^{*}$-extreme in all even order duals of $\mathcal M$.

The motivation for the theorem below comes from \cite[Theorem~2.12]{R}, where the author proved an analogous result in the context of commutative algebras without involution.
 \begin{letterthm}
 \label{t:B}
     Let $x\in \cM_1$. The following are equivalent.
     \begin{enumerate} [(i)]
         \item  $x$ is a $C^*$-extreme point of $\cM_1$.
         \item $x$ is a linear extreme point of $\cM_1$.
         \item $x$ is a strongly extreme points of $\cM_1$.
     \end{enumerate}
      In particular,  any linear extreme point of the closed unit ball of a  $C^\ast$-algebra is a strongly extreme point.
      \end{letterthm}
     A proof of Theorem \ref{t:B} can be found in Section \ref{sec:main result}.




Another notion that we consider in this section is vector-valued continuous functions space. More precisely,
let $\Omega$ be a compact Hausdorff space and let $X$ be  a Banach space. We denote by $C(\Omega, X)$  the space of all $X$-valued continuous functions, equipped with the supremum norm. It was shown in \cite{DHS} that $f\in C(\Omega, X)_1$ is a strongly extreme point if and only if $f(\omega)$ is a strongly extreme point in $X_1$ for every $\omega \in \Omega$.
 Motivated by this work,
we consider the space $C(\Omega, \mathcal A)$, where $\Omega$ is a compact Hausdorff space and $\mathcal A$ is a $C^*$-algebra. It is well known that $C(\Omega, \mathcal A)$ is a unital $C^*$-algebra. Furthermore, if $\Omega$ is infinite and $\mathcal A$ is an infinite-dimensional von Neumann algebra $\mathcal M$, then $C(\Omega, \mathcal A)$ cannot be a von Neumann algebra.
 We are interested in whether the fact that $f\in C(\Omega, \cA)_1$ is a $C^*$-extreme point implies that
$f(\omega)$ is a $C^*$-extreme point of $\cA_1$ for every $\omega \in \Omega$. We are able to answer this question when the set of isolated points is dense in $\Omega$ and $\mathcal{A}$ is a von Neumann algebra $\mathcal{M}$.





\section{Linear Extreme Points}
\label{s:ext}
 In this section, we discuss some results on linear extreme points of the closed unit ball of a  $C^*$-algebra.
In a $C^*$-algebra, the following theorem is a well-known characterization of 
the linear extreme points of the closed unit ball.
\begin{theorem}
    \label{t:lin-ext}
    \cite[Theorem~10.2]{T}
    Let $\cA$ be a  $C^*$-algebra and $\cA_1$ denote the closed unit ball of $\cA$. Then $x$ is a linear extreme point of $\cA_1$ if and only if $x$ is a partial isometry, and it satisfies 
   \begin{equation}
       \label{eq:ext-cond}
      (\bold{1}-x^*x)\cA (\bold{1}-xx^*) =\{0\}.
   \end{equation}
\end{theorem}
\begin{remark}
\label{rem:closed}
By Theorem \ref{t:lin-ext}, it is immediate to see that the set of linear
     extreme points of $\cA_1$ is norm closed.
\end{remark}
In the below theorem, we show that linear extreme points of the closed unit ball remain linear extreme in its bidual. Later in Section \ref{sec:main result}, as a consequence of Theorem \ref{t:B}, we show that any linear extreme point of $\cA_1$ is a strongly extreme point of $\cA_1$.
 \begin{theorem}
     \label{t:ext-in-bidual}
     Let $x$ be a linear extreme point of 
     $\cA_1$. Then $x$ is also a linear extreme point of 
     $\cA_1^{**}$. In particular, every linear extreme point is a weak$^{*}$-extreme point and remains  so in all the biduals. 
 \end{theorem}

 \begin{proof}
 Since $x$ is a linear extreme point of $\cA_1$, it follows from Theorem \ref{t:lin-ext} that $x$ is a partial isometry and it satisfies Equation \ref{eq:ext-cond}.
 So, to show $x$ is a linear extreme point of $\cA_1^{**}$, it remains to prove
 \[
  (\bold{1}-x^*x)\cA^{**} (\bold{1}-xx^*) =\{0\}.
 \]
 Let $y \in \mathcal{A}^{**}$ be arbitrary. If $y = 0$, then we are done. Now consider the case $y \neq 0$.  
Applying the fact that $\mathcal{A}$ is weak$^*$-dense in $\mathcal{A}^{**}$, there exists a net $\{x_\alpha\} \subset \mathcal{A}$ such that $x_\alpha$ converges to $y$ in the weak$^*$-topology.  
For fixed $b, c \in \mathcal{A}$, we know that the map $z \mapsto b z c$ is weak$^*$-continuous on $\mathcal{A}^{**}$. Thus,
$(\mathbf{1} - x^*x) x_\alpha (\mathbf{1} - xx^*) \rightarrow
(\mathbf{1} - x^*x) y (\mathbf{1} - xx^*)$
in the weak$^*$-topology. Using Equation~\ref{eq:ext-cond}, we have
$(\mathbf{1} - x^*x) x_\alpha (\mathbf{1} - xx^*) = 0$
for all $\alpha$. Hence,
$(\mathbf{1} - x^*x) y (\mathbf{1} - xx^*) = 0.$
This completes the proof.
 \end{proof}


The next theorem establishes that the quotient map preserves linear extreme points.
 \begin{theorem}
      \label{t:ideal}
     Let $\cA$ be a unital $C^*$-algebra. Let $\cI$ be a non-trivial proper 
     closed two-sided ideal in $\cA$.
     Let $x$ be a linear extreme point of 
     $\cA_1$, then $\dist(x, \cI)=1$. Moreover, $x+\cI$ is a linear extreme point of $(\cA/ \cI)_1$.
  \end{theorem}

 \begin{proof}
 We first show that $x \notin \mathcal{I}$. Suppose, to the contrary, that $x \in \mathcal{I}$. Let $y \in \mathcal{A}$ be arbitrary. Since $x$ is a linear extreme point of $\mathcal{A}_1$, by Equation~\ref{eq:ext-cond} we have
\[
y = y x x^* + x^* x y - x^* x y x x^*.
\]
Thus, $y \in \mathcal{I}$. Hence $\mathcal{A} \subset \mathcal{I}$, a contradiction, since $\mathcal{I}$ is a proper ideal. Therefore, $\mathrm{dist}(x,\mathcal{I}) > 0$.

Next, we show that $x + \mathcal{I}$ is a linear extreme point of $(\mathcal{A}/\mathcal{I})_1$. Let $\pi : \mathcal{A} \to \mathcal{A}/\mathcal{I}$ be the canonical quotient map, given by $\pi(a) = a + \mathcal{I}$. By Theorem~\ref{t:lin-ext}, $x$ is a partial isometry and satisfies Equation~\ref{eq:ext-cond}. Therefore, $\pi(x)$ is a partial isometry in $\mathcal{A}/\mathcal{I}$.

It remains to show that $\pi(x)$ satisfies Equation~\ref{eq:ext-cond}. Observe that for any $y \in \mathcal{A}$,
\begin{align*} 
& (\pi(\bold{1})-\pi(x)^*\pi(x))\pi(y)(\pi(\bold{1})-\pi(x)\pi(x)^*)\\ =& \pi((\bold{1}-x^*x)y (\bold{1}-xx^*))\\ =& \pi(0) =0. 
\end{align*}
Thus, $\pi(x)$ is a linear extreme point of 
$(\mathcal{A}/\mathcal{I})_1$. Since every linear extreme point has norm one, it follows that $\|\pi(x)\| = 1$. In particular, $\mathrm{dist}(x,\mathcal{I}) = 1$. This completes the proof.
   \end{proof}

 \section{Characterization of $C^*$-Extreme Points}
\label{sec:C*-ext}
This section contains some preliminary results on $C^*$-extreme points in 
a von Neumann algebra, which will be used in the next section to prove our main results.

Recall that to show a point is a linear extreme point, it is enough to consider the point as a convex combination of two points. The following lemma, which is an analogue of \cite[Lemma~17]{LP}, states that the same holds in the $C^*$-extreme case. The idea of the proof is similar to that of \cite[Proposition~3.2]{BDMS}.
 
   \begin{lemma}
       \label{lem:two sums}
       Let $\cA_1$ be the closed unit ball of $\cA$. The following statements are equivalent.
       \begin{enumerate} [(i)]
           \item \label{i1:for k}
           $a\in \cA_1$ is a $C^*$-extreme point of
           $\cA_1$.
           \item
           \label{i2:for 2}
           If $a=\sum_{i=1}^2t_i^*a_it_i$,  where $a_i\in \cA_1$ and $t_i$'s are invertible in $\cA$ such that $\sum_{i=1}^2t_i^*t_i=\bold{1}$, then there exist unitaries $u_i\in \cA$ such that $a_i=u_i^*au_i$ for $i=1,2$.
       \end{enumerate}
   \end{lemma}

   \begin{proof}
     $(i)\implies (ii)$ is straightforward. Now, we show   $(ii)\implies (i)$. To prove this, we need to show that whenever 
$a = \sum_{i=1}^n t_i^* a_i t_i$
is a proper $C^*$-convex combination of $a_i \in \mathcal{A}_1$, then each $a_i$ is unitarily equivalent to $a$. We proceed by induction on $n$. The case $n = 2$ holds by assumption. Suppose the result is true for $n = k$. Now consider
\[
a = \sum_{i=1}^{k+1} t_i^* a_i t_i 
= \sum_{i=1}^{k} t_i^* a_i t_i + t_{k+1}^* a_{k+1} t_{k+1},
\]
where $t_i \in \mathcal{A}$ are invertible with $\sum_{i=1}^{k+1} t_i^* t_i = \mathbf{1}$ and $a_i \in \mathcal{A}_1$. Let $t$ be the positive square root of $\sum_{i=1}^k t_i^* t_i$. Then $t$ is invertible and
\[
t^* t + t_{k+1}^* t_{k+1} = \mathbf{1}.
\]
Define
\[
x = \sum_{i=1}^k (t_i t^{-1})^* \, a_i \, (t_i t^{-1}).
\]
Then $x \in \mathcal{A}_1$ and
\[
a = t^* x t + t_{k+1}^* a_{k+1} t_{k+1}.
\]

By assumption (\ref{i2:for 2}), both $x$ and $a_{k+1}$ are unitarily equivalent to $a$. Since $x$ is a proper $C^*$-convex combination of the $a_1, \dots, a_k$, the induction hypothesis implies that $a_i$ is unitarily equivalent to $x$ for $i = 1, \dots, k$. Hence each $a_i$ is unitarily equivalent to $a$ for $i = 1, \dots, k$.
This completes the proof.
 \end{proof}

The theorem below states that unitaries are always $C^*$-extreme points of the closed unit ball in a $C^*$-algebra.

 \begin{theorem}
     \label{t:unitary-in-A}
     Let $\cA$ be a $C^*$-algebra. Then the unitaries are $C^*$-extreme points of $\cA_1$.
 \end{theorem}

 \begin{proof}
It is well known that $\mathcal{A}$ can be embedded in $B(H)$, the algebra of all bounded linear operators on a Hilbert space $H$. Thus, one can prove that unitaries are $C^*$-extreme points in $\mathcal{A}_1$ using the same argument as in the case of $B(H)$ in \cite[Proposition~24]{LP}.
 \end{proof}

The  theorem that follows is the classical Wold decomposition in $B(H)$ - the algebra of all bounded linear operators on a Hilbert space $H$. For the sake of completeness, we include a proof here, as the specific nature of the decomposition of 
$H$ will be used subsequently.
 \begin{theorem}
 \label{thm:Wold decomposition}
 Let $a$ be an isometry in $B(H)$. Let $K=\cap_{n=0}^\infty a^n H$ and $R=\range(a)^\perp$. Then the following holds.
 \begin{enumerate} [(i)]
     \item The closed subspace $K$ is $a$-invariant and $a|_K$ is unitary in $B(K)$.
     \item   The closed subspace $K^\perp$ is  $a$-invariant.
     \item For all $n,m \in \N \cup \{0\}$ with $n \neq m$, $a^n R \perp a^m R$ and $K^\perp= \oplus_{n=0}^\infty a^n R$.
 \end{enumerate}
 \end{theorem}

 \begin{proof}
 \begin{enumerate}[(i)]
     \item Using the construction of $K$, one can easily see that $K$ is $a$-invariant. Since $a$ is an isometry, it follows that $a|_K$ is injective. Now we show that $a|_K$ is surjective. Let $\eta \in K$ be arbitrary. Then $\eta \in a H$. So, there exists a unique (because $a$ is isometry) $\xi \in H$ such that $a(\xi)=\eta$. Our claim is $\xi \in K$. Since $\eta \in K= \cap_{n=0}^\infty a^n H$, it follows that there exists $\xi_{0}\in H$ such that $a^{n+1}(\xi_{0})=\eta$. In particular, $a(a^n(\xi_{0}))=\eta$. But $\xi\in H$ was a unique vector such that $a(\xi)=\eta$. We have $a^n(\xi_{0})=\xi$. Thus, $\xi \in a^n H$. Since $n\in \N$ is arbitrary, it follows that $\xi \in K$. Hence $a|_K$ is surjective. In other words, $a|_K$ is a unitary operator in  $B(K)$.

     \item Let $\xi \in K^\perp$ and $\eta \in K$ be fixed.
     We need to show that $\langle a \xi, \eta\rangle=0$. But $a|_K$ is unitary, so it is equivalent to show that $\langle a \xi, a \eta\rangle=0$, which is true because $a$ is an isometry, $\xi\in K^\perp$, and $\eta \in K$.

     \item Let $\xi, \eta \in R$ be arbitrary and $n,m \in \N \cup \{0\}$ such that $n\neq m$. Without loss of generality, assume $n>m$. Consider $\langle a^n \xi, a^m \eta \rangle$. Since $a$ is an isometry, it follows that $\langle a^n \xi, a^m \eta \rangle= \langle a^{n-m} \xi, \eta \rangle=0$ (because $a^{n-m} \xi \in \range(a)$ and $\eta \in R= \range (a)^\perp$). 
     
     Now we show that $K^\perp= \oplus_{n=0}^\infty a^n R$ which is equilvalent to show that $K= (\oplus_{n=0}^\infty a^n R)^\perp$. For that, let $\xi \in (\oplus_{n=0}^\infty a^n R)^\perp$. We need to show that $\xi \in a^n H$ for all $n \geq 0$. If not, let $N \in \N$ be the smallest number such that $\xi \notin a^N H$. Since $\xi \in a^{N-1} H$, we have $\xi= a^{N-1} \xi_0$. Observe that $\xi_0 \notin \range(a)$.
     In other words, $\xi \notin R^\perp$. Thus there exists $\theta \in R$ such that $\langle \xi_0, \theta \rangle \neq 0$ which in turn gives $\langle a^{N-1}(\xi_0), a^{N-1}(\theta) \rangle \neq 0$. 
     That is,  $\langle \xi, a^{N-1}(\theta) \rangle \neq 0$ which is a contradiction. Hence $\xi \in K$. 

Conversely, let $\xi \in K$. We need to show that $\langle \xi, a^n(\theta) \rangle=0$ for all $\theta \in R=\range(a)^\perp$ and $n\geq 0$. Since $\xi \in K$, it follows that there exists $\eta\in H$ such that $\xi=a^{n+1}(\eta)$. Thus $\langle \xi, a^n(\theta) \rangle= \langle a^{n}(a(\eta)), a^n(\theta) \rangle= \langle a(\eta), \theta \rangle$ which is zero because $a(\eta) \in \range(a)$ and $\theta \in R= \range(a)^\perp$. This completes the proof.
 \end{enumerate}
 \end{proof}

The following remark illustrates the matrix decomposition of an isometry in $B(H)$.
\begin{remark}
    \label{rem:some notations}
Let $a$  be an isometry in $B(H)$. Let $R= \range(a)^\perp$. Then by Theorem \ref{thm:Wold decomposition}, $H$ can be decomposed as $H= K \oplus K^\perp$, where $K= \cap_{n=0}^\infty a^n H$ and $K^\perp= \oplus_{n=0}^\infty a^n R$ and the matrix decomposition of $a: K \oplus K^\perp \to K \oplus K^\perp$ is given by
\[
\begin{bmatrix}
         a_1 & 0\\
         0 & a_2
     \end{bmatrix},
     \]
     where $a_1=a|_K$ is a unitary operator in $B(K)$ and $a_2=a|_{K^\perp}$ is a shift operator in $B(K^\perp)$.
     Similarly, if  $b$ is another isometry in $B(H)$. We denote $R'= \range(b)^\perp$, $L= \cap_{n=0}^\infty b^n H$ and $L^\perp= \oplus_{n=0}^\infty b^n R'$. Then the matrix decomposition of $b : L\oplus L^\perp \to L\oplus L^\perp$ is given by 
     \[
     \begin{bmatrix}
         b_1 & 0\\
         0 & b_2
     \end{bmatrix},
     \]
     where $b_1=b|_L$ is a unitary operator in $B(L)$ and $b_2=b|_{L^\perp}$ is a shift operator in $B(L^\perp)$. 
\end{remark}

The following lemmas show that the Wold decomposition, originally proved for $B(H)$, also holds in the setting of von Neumann algebras.
Perhaps this is known, but for the sake of completeness, we give a detailed proof.
Using the lemmas below, in Theorem \ref{t:siml imp unit equiv}, we  show that if two isometries in a von Neumann algebra are similar, then they are, in fact, unitarily equivalent.
In the following lemma, for the closed subspace $K$ of the Hilbert space $H$, we use the notation $\bold{1}_K$ to denote the identity of the space $B(K)$.

\vspace{.3cm}
\textbf{Convention:} In this article, we assume throughout that a von Neumann algebra $\cM$ is embedded in $B(H)$
for some Hilbert space $H$.

 \begin{lemma}
 \label{lem:uni part equi}
 Let $\cM\subset B(H)$ be a von Neumann algebra.
     Let  $a$ and $b$ be isometries in 
     $\cM$ such that $a$ and $b$ are similar, that is, there exists an invertible element $t \in \cM$ such that 
     $b=tat^{-1}$. Let $K,L,a_1,$ and $b_1$ be defined as in Remark \ref{rem:some notations} for $a$ and $b$.
     Then there exists a partial isometry $w\in \cM$ such that 
      $w|_K: K \to L$ is a unitary and   
     $b_1=wa_1w^*$.
 \end{lemma}

 \begin{proof}
     Since $a$ and $b$ are isometries in $\cM$, it follows that $a_1$ and $b_1$ are injective. Also, by construction of $K$ and $L$, one can easily show that $a_1$ is a surjective map from $K$ to $K$. Similarly  $b_1$ is surjective from $L$ to $L$. Thus, $a_1\in B(K)$ and $b_1\in B(L)$ are unitaries. We also have $b^n=ta^nt^{-1}$ for all $n\geq 0$. In other words, $b^nt=ta^n$ for all $n\geq 0$. Now, we show that $t$ maps $K$ to $L$. For that, let $\xi\in K=\cap_{n=0}^\infty a^n(H)$, then there exists $\eta\in H$ (depending on $n)$ such that $\xi=a^n\eta$. Thus $t(\xi)=ta^n(\eta)= b^nt(\eta) \in \range(b^n)$. Since $n$ is arbitrary, one has $t(\xi)\in \range(b^n)$ for all $n\geq 0$. Hence $t(\xi)\in L$. Similarly one can show that $t^{-1}$ maps $L$ to $K$. In particular, $t$ restricted to $K$ is invertible. So, the matrix decomposition of $t: K \oplus K^\perp \to L \oplus L^\perp$ is given by
     \[
     \begin{bmatrix}
         t_0 & {*}\\
         0 & {*}
     \end{bmatrix},
     \]
     where $t_{0}: K \to L$ is $t|_K$ and hence invertible. Extend $t_0$ from $K$ to $H$ by defining $t_0$ on $K^\perp$ as the zero map. Thus, we can consider $t_0 \in B(H)$.  Let $e \in \cM$ be the projection onto $K$ and $f \in \cM$ be the projection onto $L$. Then, clearly $t_0=f t e \in \cM$. Now, using the fact that $bt=ta$ and the matrix decomposition of $a, b$ and $t$, we have $b_1t_0=t_0a_1$. 
     Take the polar decomposition of $t_0=w|t_0|$, where $w$ is a partial isometry such that $w|_K : K \to L$ is unitary. By \cite[Proposition~6.1.3]{KR2}, we have $w \in \cM$.
     Applying the polar decomposition of $t_0$ in the equality $b_1t_0=t_0a_1$, one has $b_1w|t_0|= w|t_0|a_1$. This gives
     \[
     w^*b_1w|t_0|= |t_0|a_1.
     \]
     On multiplying by $a_1^*$ on the right side of the above equality, one obtains
     \[
     w^*b_1w|t_0|a_1^*= |t_0|.
     \]
     For notation simplicity, denote $w^*b_1w=b_1^{'}$. Then the above equality becomes $b_1^{'}|t_0|a_1^*= |t_0|$. This can also be written as  $b_1^{'}a_1^*a_1|t_0|a_1^*= |t_0|$. Observe that $b_1^{'}a_1^*$ is unitary and $a_1|t_0|a_1^*$ is a positive element such that $(b_1^{'}a_1^*)(a_1|t_0|a_1^*)=|t_0|$.
     Thus, it is a polar decomposition of $|t_0|$. Hence $b_1^{'}a_1^*=\bold{1}_K$ and $a_1|t_0|a_1^*=|t_0|$ which in turns gives $w^*b_1wa_1^*=\bold{1}_K$.  In particular, $w^*b_1w=a_1$. This completes the proof.  
 \end{proof}

 \begin{lemma}
     \label{lem:shift part M-vN equiv}
 Let $\cM\subset B(H)$ be a von Neumann algebra.
     Let  $a$ and $b$ be isometries in 
     $\cM$ and $t \in \cM$ be invertible such that 
     $b=tat^{-1}$.  Let $R, R{'}, K^\perp,L^\perp,a_2,$ and $b_2$ be defined as in Remark \ref{rem:some notations} for $a$ and $b$. Then the following hold:
     
     \begin{enumerate} [(i)]
         \item \label{i:p.i.from R to R'} There exists a partial isometry $v \in \cM$ such that $v^*v$ is the projection onto $\range(a)^\perp$ and $vv^*$ projection onto $\range(b)^\perp$.
         \item
         \label{i:p.i.2}
         There exists a partial isometry $s \in \cM$ such that $s^*s$ is the projection onto $K^\perp$ and $ss^*$ is the projection onto $L^\perp$.
     \end{enumerate}
 \end{lemma}

 \begin{proof}
 \begin{enumerate} [(i)]
   \item  Let $p_0\in \cM$ be the projection onto $\range(a)^\perp$. Using the fact that $b=tat^{-1}$, we see that $(t^{-1})^*$ maps $\ker(a^*)$ onto $\ker(b^*)$. Consider the operator $(t^{-1})^*p_0 \in \cM$. Since $p_0$ is the projection onto $\range(a)^\perp$ and $(t^{-1})^*$ maps $\ker(a^*)$ onto $\ker(b^*)$, it follows that $\ker((t^{-1})^*p_0)^\perp= \range(a)^\perp$ and $\range((t^{-1})^*p_0)=\range(b)^\perp$. Take the polar decomposition  $(t^{-1})^*p_0=v |(t^{-1})^*p_0|$. Note that $v \in \cM$ is a partial isometry whose initial space is $\ker((t^{-1})^*p_0)^\perp$ $= \range(a)^\perp$ and final space is $\range((t^{-1})^*p_0)=\range(b)^\perp$. This $v \in \cM$ is the required partial isometry such that $v^*v$ is the projection onto $\range(a)^\perp$ and $vv^*$ is the projection onto $\range(b)^\perp$.
   
     \item Recall that $R=\range(a)^\perp$ and $R{'}= \range(b)^\perp$. Observe that the closed subspaces $R$ and $a^n(R)$ are isomorphic via the map $a^n|_R : R \to a^n(R)$ for all $n\geq 0$. Similarly, the closed subspaces $R'$ and $b^n(R{'})$ are isomorphic for all $n\geq 0$. Note that for all $n\geq 0$, the operator  $b^nva^{*n} \in \cM$ is a partial isometry whose initial space is $a^n(R)$ and  final space is $b^n(R')$, where $v$ is defined as in  item (\ref{i:p.i.from R to R'}). Furthermore, using \cite{KR2}, the series 
    \[
    s:=\sum_{n=0}^\infty b^n va^{*n}
    \]
     converges in strong operator topology (SOT)  in $\cM$, and $s$ is the partial isometry such that $s^*s$ is the projection onto $K^\perp$ and $ss^*$ is the projection onto $L^\perp$. This completes the proof.
 \end{enumerate}
 \end{proof}

We now prove the main result of this section.
 \begin{theorem}
     \label{t:siml imp unit equiv}
      Let  $a, b \in \cM \subset B(H)$ be isometries. 
     Let $t \in \cM$ be  invertible 
     such that 
     $b=tat^{-1}$. There exists a unitary $u \in \cM$ such that $b=uau^*$.
 \end{theorem}

 \begin{proof}
     Using Remark \ref{rem:some notations}, we obtain the matrix decomposition of $a: K \oplus K^\perp \to K \oplus K^\perp$ as $\begin{bmatrix}
         a_1 & 0\\
         0 & a_2
     \end{bmatrix}$ and the matrix decomposition of $b: L\oplus L^\perp\to L\oplus L^\perp$ as $\begin{bmatrix}
         b_1 & 0\\
         0 & b_2
     \end{bmatrix}$. Define $u:  K \oplus K^\perp \to  L\oplus L^\perp$ as 
     \[
     \begin{bmatrix}
         w & 0\\
         0 & s
     \end{bmatrix},
     \] 
     where $w\in \cM$ is defined as in Lemma \ref{lem:uni part equi} and $s \in \cM$ is defined as in Lemma \ref{lem:shift part M-vN equiv}. Since $w:K \to L$ is unitary and $s: K^\perp \to L^\perp$ is unitary, it follows that $u \in B(H)$ is unitary.  Moreover, $u\in \cM$. Again using Lemmas \ref{lem:uni part equi}, \ref{lem:shift part M-vN equiv}, we have  that $b=uau^*$.
 \end{proof}

\section{Geometric Aspects of $C^*$-Extreme Points}
\label{sec:main result}
This section contains  proofs of Theorems~\ref{t:A} and~\ref{t:B}. 
It is well known that the linear extreme points of $B(H)_1$ are precisely the isometries and coisometries. 
In \cite[Theorem~1.1]{HMP}, Hopenwasser et al.\ characterized the $C^*$-extreme points of the closed unit ball $B(H)_1$. 
In particular, they showed that the $C^*$-extreme points of $B(H)_1$ coincide with its linear extreme points. 
In this section, our aim is to characterize the $C^*$-extreme points of the closed unit ball in a general von Neumann algebra. 
Moreover, we show that, in the setting of a von Neumann algebra, the notions of linear extreme, strongly extreme, and $C^*$-extreme points all coincide.

The idea of the theorem below is to apply a Wold decomposition type theorem in the setting of von Neumann algebras, along the same lines as those used in \cite[Theorem~1.1]{HMP} in the case of $B(H)$.

\begin{theorem}
\label{t:iso-imp-C*ext-in-vN}
    Let $\cM \subset B(H)$ be a von Neumann algebra. Isometries and coisometries are the $C^*$-extreme points of 
    $\cM_1$.
\end{theorem}

\begin{proof}
Let $v\in \cM_1$ be an isometry.
Using the alternate characterization of a $C^*$-extreme point given in  Lemma \ref{lem:two sums},
suppose 
    \[
    v= t_1^*a_1t_1+t_2^*a_2t_2,
    \]
    where $a_1, a_2\in \cM_1$ and $t_1, t_2$ are invertible elements in $\cM$ such that $\sum_{i=1}^2t_i^*t_i=\bold{1}$. 
Define $t: H \to H \oplus H$ via $t\xi= (t_1\xi, t_2 \xi)$ for each $\xi \in H$. Then $t^*t= t_1^*t_1+t_2^*t_2=\bold{1}$. Thus, $t$ is an isometry. Let $p=tt^*$ be the projection onto $\range(t)$. Let $a=\begin{bmatrix}
    a_1 & 0\\
    0 & a_2
\end{bmatrix}$. Then $t^*at= t_1^*a_1t_1+t_2^*a_2t_2=v$. Now $v^*v= (t^*at)^*(t^*at)= t^*a^*tt^*at=t^*a^*pat$. Since $0\leq p\leq \bold{1}$, we have $a^*pa\leq a^*a$. Therefore,
\begin{equation}
\label{eq:for v}
v^*v= t^*a^*pat \leq t^*a^*at.
\end{equation}
Note that 
\begin{equation}
\label{eq:for-t-and-a}
    t^*a^*at= t_1^*a_1^*a_1t_1+t_2^*a_2^*a_2t_2.
    \end{equation}
    Since $\|a_i\|\leq 1$, we get $0\leq a_i^*a_i\leq \bold{1}$. Thus,
    \begin{equation}
    \label{eq:equality}
      t_1^*a_1^*a_1t_1+t_2^*a_2^*a_2t_2 \leq t_1^*t_1+t_2^*t_2=\bold{1}.  
    \end{equation}

    Combine Equations (\ref{eq:for v}), (\ref{eq:for-t-and-a}) and (\ref{eq:equality}), we obtain
    \begin{equation*}
        \bold{1}= v^*v\leq  t_1^*a_1^*a_1t_1+t_2^*a_2^*a_2t_2 \leq \bold{1}.
    \end{equation*}

So, equality holds throughout, and in particular, $t_1^*a_1^*a_1t_1+t_2^*a_2^*a_2t_2= \bold{1}= t_1^*t_1+t_2^*t_2$. This gives 
\begin{equation}
\label{eq:positive}
    t_1^*(\bold{1}-a_1^*a_1)t_1+  t_2^*(\bold{1}-a_2^*a_2)t_2=0
\end{equation}

Observe that each summand of the above equation (\ref{eq:positive}) is positive. Hence, each summand has to be zero, that is, $t_i^*(\bold{1}-a_i^*a_i)t_i=0$ for $i=1,2$. Since $t_i$'s are invertible, it follows that $\bold{1}-a_i^*a_i=0$, that is, $a_i^*a_i=\bold{1}$ for $i=1,2$. Thus, $a_1$ and $a_2$ are isometries in $\cM$. 

Now we show that $a_1$ and $a_2$ are similar to $v$. 
Again, using Equations (\ref{eq:for v}), (\ref{eq:for-t-and-a}) and (\ref{eq:equality}), one has 
$t^*a^*pat=t^*a^*at$, that is, $t^*a^*(\bold{1}-p)at=0$. Note that $a^*(\bold{1}-p)a \geq 0$ and 
\[
0=t^*a^*(\bold{1}-p)at= ((\bold{1}-p)^{1/2}at)^* ((\bold{1}-p)^{1/2}at).
\]
Thus, $((\bold{1}-p)^{1/2}at)=0$. This gives $(\bold{1}-p)at=0$, that is, $at=pat = tt^*at$. Therefore $\range(at)\subseteq \range(t)$. Now, by 
\cite[Theorem~1]{Dou}, there exists an operator $b\in B(H)$ such that $at=tb$. Thus, 
\[
t^*at=t^*tb=b.
\]
But $t^*at=v$, hence $v=b$. So, we have $at=tv$, that is, $(a_1t_1, a_2t_2)=(t_1v,t_2v)$. In particular, $a_1t_1=t_1v$ and $a_2t_2=t_2v$. Equivalently, 
\[
a_1=t_1vt_1^{-1}~~~ \text{and} ~~~  a_2=t_2vt_2^{-1}.
\]
Hence, $a_1$ and $a_2$ are similar to $v$.   
  Now using Theorem \ref{t:siml imp unit equiv}, there exist unitaries $u_1$ and $u_2$ in $\cM$ such that $a_1= u_1vu_1^*$ and  $a_2= u_2vu_2^*$. Therefore, $v$ is a $C^*$-extreme point of $\cM_1$. Similarly, one can show that every coisometry is a $C^*$-extreme point of $\cM_1$.
 This completes the proof.
  \end{proof}


Recall that a Banach space $X$ is \textbf{uniformly convex} if given $\epsilon >0$ there is a $\delta >0$ such that whenever $\|x\|=1=\|y\|$ and $\|x-y\| \geq \epsilon$, then $\|x+y\|\leq 2(1-\delta)$. Moreover, any uniformly convex space is always a reflexive space. See \cite{D}.
For $1<p<\infty$, $L_p(\mu)$ spaces are well-known examples of uniformly convex spaces.
 Now we prove a geometric analogue of 
 Theorem \ref{t:iso-imp-C*ext-in-vN} in the context of strongly extreme points on a uniformly convex space $X$. We recall that $T \in B(X)$ is a coisometry if $T^\ast \in B(X^\ast)$ is an isometry.
  \begin{theorem}
     \label{t:st-ext}
      Let $X$ be a uniformly convex space. Let $T \in B(X)_1$ be an isometry. Then $T$ is a strongly extreme point of $B(X)_1$. Similarly, if $X^*$ is uniformly convex, then every coisometry in $B(X)$ is also a strongly extreme point of $B(X)_1$.
 \end{theorem}

  \begin{proof}
      Let $R_n$ and $S_n$ be two sequences in $B(X)$ such that $\frac{1}{2}(R_n+S_n) \to T$. Let $\xi\in X$ be such that $\|\xi\|=1$. Since $T$ is an isometry and $\|\xi\|=1$, it follows that $\|T\xi\|=1$. Thus $T\xi$ is a linear extreme point of $X_1$.
      Now $T(\xi)= \lim_{n \to \infty}\frac{1}{2}(R_n(\xi)+S_n(\xi))$. Since $X$ is uniformly convex, it is easy to see that $\|R_n(\xi)-S_n(\xi)\| \to 0$ uniformly on all $\xi\in X$ with $\|\xi\|=1$. Thus $\|R_n-S_n\| \to 0$. Hence, $T$ is a strongly extreme point of  $B(X)_1$. In the case, when $X^*$ is uniformly convex, the similar result for coisometry is easy to see using the fact that the map  $S \mapsto S^*$ is a surjective isometry from $B(X)$ to $B(X^*)$.
       This completes the proof.
  \end{proof}

  \begin{remark}
  \label{rem:st.ext-indpt-T}
     Note that strong extremality depends only on the uniform convexity of 
$X$, and not on the specific operator 
$T$ we are using.
  \end{remark}

Next set of results get substantially strengthened by Theorem \ref{t:B} and are included here only as an illustration of the techniques developed here.
  \begin{corollary}     \label{c:st.ext-vN}
      Let $\cM \subset B(H)$ be a von Neumann algebra. Isometries and coisometries are strongly extreme points of $\cM_1$.
  \end{corollary}
   \begin{proof}
       Follows immediately from Theorem \ref{t:st-ext}.
   \end{proof}

  \begin{proposition}
      \label{p:st.ext-in-d.sum}
  Let $\{H_\alpha\}_{\alpha \in \Delta}$ be a family of Hilbert spaces.
  Consider $\oplus_\alpha B(H_\alpha)$ $(\ell_\infty$-sum).
      Let 
      $T  $ be a linear extreme point of $(\oplus B(H_\alpha))_1$, then $T$ is a strongly extreme point of $(\oplus B(H_\alpha))_1$. Moreover, if $\cM$ is a von Neumann algebra whose predual has the Radon-Nikodym Property (RNP); then every linear extreme point of $\cM_1$ is a strongly extreme point of $\cM_1$.
  \end{proposition}
  
   \begin{proof}
   If $T $ is a linear extreme point of $(\oplus B(H_\alpha))_1$, then it is easy to see that each $T(\alpha)$ is a linear extreme point of $B(H_\alpha)_1$. More precisely, each $T(\alpha)$ is an isometry or a coisometry in $B(H_\alpha)$. By  
  Theorem \ref{t:st-ext} and Remark \ref{rem:st.ext-indpt-T}, it is immediate that $T$ is a strongly extreme point of $(\oplus B(H_\alpha))_1$.  For the case, when $\cM$ is a von Neumann algebra whose predual has RNP, then by \cite[Theorem]{C}, we have $\cM= \oplus_\alpha B(H_\alpha)$ ($\ell_\infty$-sum). Applying the same argument as above, we conclude that every linear extreme point of $\cM_1$ is also a strongly extreme point of $\cM_1$.
   \end{proof}
  
Below, we extend Proposition \ref{p:st.ext-in-d.sum} to the $C^*$-extreme case. Specifically, we show that an element in the $\ell_\infty$-direct sum of two $C^*$-algebras is $C^*$-extreme if and only if it is $C^*$-extreme in each coordinate. Since we deal with several $C^*$-algebras in the next theorem, we index their identities by the respective $C^*$-algebras.
  \begin{theorem}
\label{t:componentwise C*ext}
 Let $\cA, \cB$ and $\cC$ be  $C^*$-algebras such that $\cC=\cA \oplus \cB$ ($\ell_\infty$-sum). Then  $x=(a,b)$ 
 is a $C^*$-extreme point of 
 $\cC_1$ if and only if $a$ is a $C^*$-extreme point of 
 $\cA_1$ and $b$ is a $C^*$-extreme point of 
 $\cB_1$.
  \end{theorem}

  \begin{proof}
 Let $x=(a,b)$ be a $C^*$-extreme point of $\cC_1$.  Let $a=t_1^*a_1t_1+t_2^*a_2t_2$ where $a_1, a_2\in \cA_1$ and $t_1,t_2 \in \cA$ are invertible such that $t_1^*t_1+t_2^*t_2=\bold{1}_\cA$. Now $x=(a,b)$ can be written as 
 \begin{equation}
 \label{eq:properC_comb}
     (a,b)= \left(t_1^*,\frac{1}{\sqrt{2}}\bold{1}_\cB\right)(a_1,b) \left(t_1,\frac{1}{\sqrt{2}}\bold{1}_\cB\right)+ \left(t_2^*,\frac{1}{\sqrt{2}}\bold{1}_\cB\right)(a_2,b) \left(t_2,\frac{1}{\sqrt{2}}\bold{1}_\cB\right).
  \end{equation}
  Observe that Equation \ref{eq:properC_comb} represents $x$ as  a proper $C^*$-convex combination of $(a_1,b)$, $(a_2,b) \in \cC_1$. Since $(a,b)$ is a $C^*$-extreme point of $\cC_1$, it follows that there exist unitaries $u=(u_1,u_2),v=(v_1,v_2) \in \cC$ such that 
  \[
  (a_1,b)= (u_1^*,u_2^*)(a,b) (u_1,u_2),~~ \text{and}~~ (a_2,b)= (v_1^*,v_2^*)(a,b) (v_1,v_2).
  \]
  Clearly, $u_1$ and  $u_2$ are unitaries in $\cA$ and $v_1$ and $v_2$ are unitaries in $\cB$.
 Thus $a_1=u_1^*au_1$ and $a_2=v_1^*av_1$. Hence $a$ is a $C^*$-extreme point of $\cA_1$. Similarly, one can show that $b$ is a $C^*$-extreme point of $\cB_1$. 

 Conversely, let $a$ and $b$ be $C^*$-extreme points of $\cA_1$ and $\cB_1$ respectively. Let 
 \[
 (a,b)= (s_1^*,s_2^*)(a_1,b_1)(s_1,s_2)+ (t_1^*,t_2^*)(a_2,b_2)(t_1,t_2),
 \] 
 where $(a_1,b_1),(a_2,b_2)\in \cC_1$ and $(s_1,s_2), (t_1,t_2)$ are invertible in $\cC$ such that $(s_1,s_2)^*(s_1,s_2)+(t_1,t_2)^*(t_1,t_2)=\bold{1}_\cC$.
Thus 
\[
a=s_1^*a_1s_1+t_1^*a_2t_1~~ \text{and} ~~ b=s_2^*b_1s_2+t_2^*b_2t_2,
\]
 where $s_1,t_1 \in \cA$ and $s_2,t_2\in \cB$ are invertible satisfying $s_1^*s_1+t_1^*t_1=\bold{1}_\cA$ and $s_2^*s_2+t_2^*t_2=\bold{1}_\cB$, and where $a_1,a_2\in \cA_1$ and $b_1, b_2 \in \cB_1$. Since $a$ and $b$ are $C^*$-extreme points of $\cA_1$ and $\cB_1$, respectively, it follows that there exist unitaries $u_1,u_2\in \cA$ and $v_1, v_2\in \cB$ such that $a_i=u_i^*au_i$ and $b_i=v_i^*bv_i$ for $i=1,2$. Define $w_1=(u_1,v_1), w_2=(u_2,v_2)\in \cC$. Clearly, $w_1, w_2 \in \cC$ are unitaries and  $(a_i,b_i)= w_i^*(a,b)w_i$ for $i=1,2$. Thus $(a,b)$ is a $C^*$-extreme point of $\cC_1$. This completes the proof.
 \end{proof}

The following corollary is a  continuous version of Theorem \ref{t:componentwise C*ext}. Recall that $C(\Omega,\cA)$ denotes the unital $C^*$-algebra, where $\Omega$ is a compact Hausdorff space and $\cA$ is a $C^*$-algebra.
 \begin{corollary}
     \label{c:C_extonisopt}
 Let $f$ be a $C^*$-extreme point
 of 
 $C(\Omega,\cA)_1$. Let $\omega_0\in \Omega$ be an isolated point. Then, $f(\omega_0)$ is a $C^*$-extreme point of 
 $\cA_1$.    
 \end{corollary}

 \begin{proof}
 First, observe that $C(\Omega,\cA)$ is $*$-isomorphic to $C(\Omega_0,\cA) \oplus \cA$ ($\ell_\infty$-sum), where $\Omega_0=\Omega\setminus\{\omega_0\}$, under the  isomorphism $f \mapsto (f|_{\Omega_0}, f(\omega_0))$.  
 Let $f\in C(\Omega,\cA)_1$ be $C^*$-extreme, then by the above $*$-isomorphism, it follows that the corresponding element $(f|_{\Omega_0}, f(\omega_0)) \in (C(\Omega_0,\cA) \oplus_ \cA)_1$ 
 is $C^*$-extreme. Now using Theorem \ref{t:componentwise C*ext}, we have $f|_{\Omega_0}$ and $f(\omega_0)$ are $C^*$-extreme points of $C(\Omega_0,\cA)_1$ and $\cA_1$, respectively. This completes the proof.
 \end{proof}

Later,
in Proposition \ref{c:vN}, we show that if the set of isolated points is dense in $\Omega$ and $\mathcal A$ is a von Neumann algebra $\mathcal M$, then any $C^{*}$-extreme point $f$ of $C(\Omega,\mathcal M)_{1}$ satisfies that $f(\omega)$ is a $C^{*}$-extreme point of $\mathcal M_{1}$ for every $\omega \in \Omega$.


In order to prove Theorem~~ \ref{t:A}, we need the following lemma which is perhaps well-known. 
\begin{lemma}
    \label{lem:cont-imp-ave-of-p.i}
Let $x\in \cM_1$. There exist partial isometries $v_1, v_2 \in \cM_1$ such that
\begin{equation*}
    x= \frac{v_1+v_2}{2}.
\end{equation*}
\end{lemma}

    \begin{proof}
 Let $x=v|x|$ be the polar decomposition of $x$, where $v \in \cM$ is a partial isometry whose initial space is $\ker(x)^\perp$ and final space is $\overline{\range(x)}$. 
 Let $u_1= |x|+i\sqrt{\bold{1}-|x|^2}$ and $u_2=|x|-i\sqrt{\bold{1}-|x|^2}$. Since $|x|$ is a contraction, it follows that $u_1$ and $u_2$ are well-defined unitaries in $\cM$, and 
    $|x|=\frac{u_1+u_2}{2}$. Thus, $x= \frac{vu_1+vu_2}{2}$. Define $v_1=vu_1$ and $v_2=vu_2$. Thus, $x= \frac{v_1+v_2}{2}$.  Clearly, each $v_i \in \cM$ is a  partial isometry for $i=1,2$ 
     whose initial space is $\ker(x)^\perp$ and final space is $\overline{\range(x)}$.
     \end{proof}



Now we recall the Comparison Theorem from \cite{KR2}.  We refer to \cite{KR,KR2} for standard terminology, notation, and basic definitions concerning von Neumann algebras.
\begin{theorem} 
    \label{t:comp thm}
    (\cite[Theorem ~6.2.7]{KR2})
    Let $\cM$ be a von Neumann algebra. Let 
    $e$ and $f$ be projections in $\cM$.
There exist central projections $p,q,r \in \cM$ such that $p+q+r=\bold{1}$ with the property that $pe \sim pf$, and, if $q\neq 0$, then $qe \prec qf$. If $r\neq 0$, then $rf \prec re$.
\end{theorem}
Using the above Comparison Theorem,
we have the following important observation for a contraction in a von Neumann algebra.
 \begin{theorem}
     \label{t:comp-thm-on-contr}
 Let $x\in \cM_1$. 
 There exist partial isometries $x_1,x_2 \in \cM$ and central projections $p_1,p_2,p_3 \in \cM$  such that $p_1+p_2+p_3=\bold{1}$,  $x=\frac{x_1+x_2}{2}$,
 and,
 the following holds:
 \begin{enumerate} [(i)]
     \item
     \label{i1}
     either $p_1=0$, or $p_1x_1$ and $p_1x_2$ are unitaries in $p_1\cM$,
     \item  \label{i2} either $p_2=0$, or $p_2x_1$ and $p_2x_2$ are  non-unitary isometries in $p_2\cM$, and,
     \item  \label{i3} either $p_3=0$, or $p_3x_1$ and $p_3x_2$ are non-unitary coisometries in $p_3\cM$.
 \end{enumerate}
 \end{theorem}

 \begin{proof}
 Let $e\in \cM$ be the projection on $\ker(x)$  and $f \in \cM$ be the projection onto ${\range(x)}^\perp$.  Applying
 Theorem \ref{t:comp thm}
 on projections $e$ and $f$, we obtain  central projections $p_1,p_2,p_3$ such that $p_1+p_2+p_3=\bold{1}$ and there exist $w_1 \in p_1\cM, w_2 \in p_2\cM, w_3 \in p_3\cM$ with the properties that 
 \begin{enumerate} [(i)]
     \item either $p_1=0$, or $w_1^*w_1=p_1e$ and $w_1w_1^*=p_1f$,
     \item either $p_2=0$, or $w_2^*w_2=p_2e$ and $w_2w_2^*<p_2f$, and,
     \item either $p_3=0$, or $w_3^*w_3<p_3e$ and $w_3w_3^*=p_3f$.
 \end{enumerate}

For $i=1,2,3$; if $p_i=0$, we take $w_i=0$ . 
 Since $x\in \cM_1$, by Lemma \ref{lem:cont-imp-ave-of-p.i}, there exist partial isometries $v_1,v_2 \in \cM$ whose initial space is $\ker(x)^\perp$ and  final space is $\overline{\range(x)}$, and such that $x=\frac{v_1+v_2}{2}$.
 Note that the matrix decomposition of $v_1,v_2:\ker(x) \oplus \ker(x)^\perp \to {\range(x)}^\perp \oplus \overline{\range(x)}$ are given by
 \begin{equation*}
       v_1= \begin{bmatrix}
         0 & 0\\
         0 & v_1
     \end{bmatrix} ~~~ \text{and} ~~~ v_2=\begin{bmatrix}
         0 & 0\\
         0 & v_2
     \end{bmatrix}.
 \end{equation*}

 Let $w=w_1+w_2+w_3$. Observe that $w\in \cM$ is a partial isometry such that $w^*w \leq e$ and $ww^*\leq f$. 
 Define $x_1,x_2: \ker(x) \oplus \ker(x)^\perp \to {\range(x)}^\perp \oplus \overline{\range(x)}$ as
 \begin{equation}
 \label{eq:C*-ext}
     x_1= \begin{bmatrix}
        w & 0\\
         0 & v_1
     \end{bmatrix} ~~~ \text{and} ~~~ x_2=\begin{bmatrix}
         -w & 0\\
         0 & v_2
     \end{bmatrix}.
 \end{equation}
Clearly, $x=\frac{x_1+x_2}{2}$. 
Recall that  $w_1 \in p_1\cM, w_2 \in p_2\cM$, and $w_3 \in p_3\cM$. Therefore, the matrix representations of $p_1x_1$ and $p_1x_2$ are
\[
p_1x_1=\begin{bmatrix}
         w_1 & 0\\
         0 & p_1v_1
     \end{bmatrix}~~~ \text{and}~~~ p_1x_2=\begin{bmatrix}
         -w_1 & 0\\
         0 & p_1v_2
     \end{bmatrix}.
     \]
Since  $p_1$ is a central projection, $w_1$ is a unitary from $\ker(x)$ to ${\range(x)}^\perp$, and 
$v_1,v_2$ are unitaries from $\ker(x)^\perp$ to $\overline{\range(x)}$, it follows that if $p_1\neq 0$, then $p_1x_1$ and $p_1x_2$ are unitaries in $p_1\cM$. Similarly, if $p_2\neq 0$, then $p_2x_1$ and  $p_2x_2$ are non-unitary isometries in $p_2\cM$, and if $p_3 \neq 0$, then $p_3x_1$ and  $p_3x_2$ are non-unitary coisometries in $p_3\cM$. This completes the proof.
 \end{proof}

Now we prove Theorem \ref{t:A}. 
  \begin{proof}
     Let $x\in \cM_1$ be a $C^*$-extreme point of $\cM_1$. 
     By  Theorem \ref{t:comp-thm-on-contr}, there exist partial isometries  $x_1$ and  $x_2$, defined as  in the proof of Theorem \ref{t:comp-thm-on-contr} such that
     $x=\frac{x_1+x_2}{2}$.
     Since $x$ is $C^*$-extreme, it follows that $x$ is unitarily equivalent to $x_1$ and $x_2$. That is, there exist unitaries $u_1, u_2 \in \cM$ such that
     \[
     x=u_1^*x_1u_1, ~~~\text{and}~~~ x=u_2^*x_2u_2.
     \]
     Again, using Theorem \ref{t:comp-thm-on-contr}, there exist central projections $p_1,p_2,p_3 \in \cM$ satisfying $p_1+p_2+p_3=\bold{1}$ and the conditions (\ref{i1}),(\ref{i2}) and (\ref{i3}) hold.
      Now, 
      $p_1x=p_1u_1^*x_1u_1=p_1u_1^*p_1x_1p_1u_1$. 
      From Condition (\ref{i1}) (Theorem \ref{t:comp-thm-on-contr}),
      one has either $p_1=0$, or  $p_1x_1$ is a unitary in $p_1\cM$. Also, if $p_1 \neq 0$, then $p_1u_1$ is a unitary in $p_1\cM$. Thus, we have either $p_1=0$ or $p_1x$ is a unitary in $p_1\cM$. 
      Similarly, one can show that either $p_2=0$, or $p_2x$ is a non-unitary isometry in $p_2\cM$, 
      and either $p_3=0$, or $p_3x$ is a non-unitary coisometry in $p_3\cM$. Therefore, we have the required central projections that satisfy conditions (\ref{i1:u}), (\ref{i2:iso}) and (\ref{i3:coiso}). 

      Conversely, if there exist central projections $p_1,p_2,p_3 \in \cM$ such that  $p_1+p_2+p_3=\bold{1}$ and the conditions (\ref{i1:u}), (\ref{i2:iso}) and (\ref{i3:coiso}) hold. 
      Then, by Theorem \ref{t:iso-imp-C*ext-in-vN}, we have each $p_ix$ is a $C^*$-extreme point of
      the closed unit ball of 
      $(p_i\cM)_1$ for $i=1,2,3$. On applying Theorem \ref{t:componentwise C*ext}, we obtain $x=p_1x+p_2x+p_3x$ is a $C^*$-extreme point of $\cM_1$.  This completes the proof.
      \end{proof}

Similar to the case in 
$B(H)$ where every contraction is an average of an isometry and a coisometry, as an application of Theorem \ref{t:A},  we have an analogous result for von Neumann algebras.
 \begin{corollary}
     \label{c:KMT}
     Let $x\in \cM_1$. There exist $C^*$-extreme points $x_1$ and $x_2$ of $\cM_1$ such that $x=\frac{x_1+x_2}{2}$.
 \end{corollary}

 \begin{proof}
 If $x = 0$, then take $x_1 = \mathbf{1}$ and $x_2 = -\mathbf{1}$. Thus, we have 
  $x = \frac{x_1 + x_2}{2}$.
Now suppose $x \neq 0$. Define $x_1$ and $x_2$ as in Equation~\ref{eq:C*-ext}. By Theorem~\ref{t:comp-thm-on-contr}, both $x_1$ and $x_2$ satisfy conditions~(\ref{i1}), (\ref{i2}), and (\ref{i3}), and moreover 
  $x = \frac{x_1 + x_2}{2}$.
From Theorem~\ref{t:A}, it follows that $x_1$ and $x_2$ are $C^*$-extreme points of $\cM_1$. This completes the proof.
 \end{proof}

As another application of Theorem~\ref{t:A}, we obtain the following theorem, which characterizes certain von Neumann algebras in terms of their $C^*$-extreme points.

 \begin{theorem}
    \label{t:C*ext-in-subclass}
    Let $\cM \subset B(H)$ be a von Neumann algebra. Then,  the only
    $C^*$-extreme points of $\cM_1$ are isometries and coisometries
    if and only if  
        $\cM$ is one of the following:
        \begin{enumerate} [(i)]
            \item a finite von Neumann algebra,
            \item a properly infinite factor, or
            \item a direct sum of a finite von Neumann algebra and a properly infinite factor. 
        \end{enumerate}
\end{theorem}

\begin{proof}
First, assume that  the only $C^*$-extreme points of $\cM_1$ are isometries and coisometries. If $\cM$ is a finite von Neumann algebra or a properly infinite factor, then there is nothing to prove. Now,
 we consider the case 
 $\cM= \cN \oplus \cR$ ($\ell_\infty$-sum)
 such that $\cN$ is a finite von Neumann algebra and
  $\cR$ is a properly infinite von Neumann algebra. If $\cR$ is a properly infinite factor, we are done. 
  So, if possible, assume $\cR$ is not a factor.
  Then, there exists a central projection $p\in \cR$ such that $p \neq 0,\bold{1}_{\cR}$, where $\bold{1}_{\cR}$ denote the identity of von Neumann algebra $\cR$. Thus $\bold{1}_{\cR}-p \neq 0$. Observe that $\cR= p \cR \oplus (\bold{1}_{\cR}-p)\cR$ ($\ell_\infty$-sum). In this case, we can choose an isometry $v\in p \cR$ and a coisometry $w\in (\bold{1}_{\cR}-p)\cR$ such that neither $v$ nor $w$ is a unitary, then using Theorem \ref{t:iso-imp-C*ext-in-vN} and Theorem \ref{t:componentwise C*ext}, we have $u=(v,w)$ is a $C^*$-extreme point of $\cR_1$ and hence $(\bold{1}_{\cN}, u)$ is a $C^*$-extreme point of $\cM_1$, where $\bold{1}_{\cN}$ denote the identity of von Neumann algebra $\cN$. But note that $(\bold{1}_{\cN}, u)$ is neither an isometry nor a coisometry in $\cM$, which is a contradiction. 
 
Conversely,  let $r \in \cM$ be a central projection such that either $r=0$ or $r\cM$ is a finite von Neumann algebra and $\bold{1}-r=0$ or
 $(\bold{1}-r)\cM$ is a properly infinite factor. Note that $\cM= r\cM \oplus (\bold{1}-r)\cM$ ($\ell_\infty$-sum). 
 Let $x \in \cM_1$ be a $C^*$-extreme point of $\cM_1$.
 Then by Theorem \ref{t:A}, there exist central projections $p_1,p_2,p_3 \in \cM$ such that $p_1+p_2+p_3=\bold{1}$ and conditions (\ref{i1:u}), (\ref{i2:iso}), and (\ref{i3:coiso}) holds. If $\mathbf{1} - r = 0$, then $\mathcal{M}$ is a finite von Neumann algebra. In this case, by conditions (\ref{i1:u}), (\ref{i2:iso}), and (\ref{i3:coiso}), we have that $x$ is a unitary in $\mathcal{M}$. 
If $r=0$, then $\cM$ is a properly infinite factor. Hence one of the $p_i$'s is $\bold{1}$.
Thus, again by conditions (\ref{i1:u}), (\ref{i2:iso}), and (\ref{i3:coiso}), we get $x$ is either isometry or coisometry. 

Lastly, we consider the case when $r\neq 0$ and $\bold{1}-r \neq 0$. More precisely, $\cM$ is a direct sum of a finite von Neumann algebra $r\cM$ and a properly infinite factor $(\bold{1}-r)\cM$. By condition (\ref{i2:iso}), we have either $p_2=0$ or $p_2x$ is an isometry in $p_2\cM$. Since $r$ is a central projection, it follows that $p_2rx$ is an isometry in $p_2r\cM$. Observe that $p_2r\cM$ is a finite von Neumann algebra because $r\cM$ is so. Thus, either $p_2=0$ or $p_2rx$ is a unitary in $p_2r\cM$. Similarly, by condition (\ref{i3:coiso}), either $p_3=0$ or $p_3rx$ is a unitary in $p_3r\cM$. Define $q=p_1+p_2r+p_3r$. Then $p_1x+p_2rx+p_3rx$ is unitary in $q\cM$. Now note that $p_2(\mathbf{1} - r)$ and $p_3(\mathbf{1} - r)$ are mutually orthogonal central projections in $(\mathbf{1} - r)\cM$ and $(\mathbf{1} - r)\cM$ is a properly infinite factor. So either $p_2(\mathbf{1} - r)=0$ or $p_3(\mathbf{1} - r)=0$. If $p_2(\mathbf{1} - r)=0$, then $q+p_3(\mathbf{1} - r)=\bold{1}$. 
Now, if $p_3\neq 0$, then by condition (\ref{i3:coiso}), $p_3(\mathbf{1} - r)x$ is a coisometry in $p_3(\mathbf{1} - r)\cM$. Because $qx$ is unitary in $q\cM$ and $p_3(\mathbf{1} - r)x$ is a coisometry in $p_3(\mathbf{1} - r)\cM$, we have $x=qx+p_3(\mathbf{1} - r)x$ is coisometry in $\cM$. Similarly, when $p_3(\mathbf{1} - r)=0$, one has $x=qx+p_2(\mathbf{1} - r)x$ is an isometry in $\cM$. This completes the proof.
\end{proof}

 Now we prove Theorem \ref{t:B}.
\begin{proof}
 $(i) \implies (ii)$:   Let $x$ be a $C^*$-extreme point of $\cM_1$. By Theorem \ref{t:A}, we have $x$ is a partial isometry. Using Theorem \ref{t:lin-ext}, it remains to verify that $x$ satisfies Equation \ref{eq:ext-cond}. To that end, let $e=x^*x$ and $f=xx^*$. Then using
  conditions (\ref{i1:u}), (\ref{i2:iso}), and (\ref{i3:coiso}) from Theorem \ref{t:A}, we obtain that 
 \begin{align*}
     \text{if}~~ p_1\neq 0,~~ \text{then}~~ p_1e=p_1 ~~\text{and}~~ p_1f=p_1,\\
      \text{if}~~ p_2\neq 0,~~ \text{then}~~ p_2e=p_2 ~~\text{and}~~ p_2f\neq p_2,\\
    \text{if}~~ p_3\neq 0,~~ \text{then}~~ p_3e \neq p_3 ~~\text{and}~~ p_3f=p_3.
 \end{align*}

 Thus, 
 \[
 (\bold{1}-e)= (\bold{1}-e)p_1+(\bold{1}-e)p_2+(\bold{1}-e)p_3= (\bold{1}-e)p_3.
 \]
 Similarly,
 \[
 (\bold{1}-f)= (\bold{1}-f)p_1+(\bold{1}-f)p_2+(\bold{1}-f)p_3= (\bold{1}-f)p_2.
 \]
 Fix $y\in \cM$. Using the fact that the central projections 
 $p_i$'s are mutually orthogonal, we get
 \begin{align*}
     (\bold{1}-e)y(\bold{1}-f)= (\bold{1}-e)p_3y(\bold{1}-f)p_2=0.
 \end{align*}
 Hence, $x$ is a linear extreme point of $\cM_1$.

 $(ii) \implies(i)$: Since $x$ is a linear extreme point of $\cM_1$. By Theorem \ref{t:lin-ext}, $x$ is a partial isometry, and it satisfies Equation \ref{eq:ext-cond}.
 Let $e=x^*x$ and $f=xx^*$. Let $p_3$ denote the central support of $(\bold{1}-e)$ and $p_2$ denote the central support of $(\bold{1}-f)$. 

 \textbf{Claim:} \textbf{$p_2p_3=0$. }

 Assume for contradiction that $p:=p_2p_3\neq 0$. Then $p$ is a nonzero central projection. Because $p\leq p_3$ and $p\neq 0$, we have $p(\bold{1}-e)\neq 0$. Similarly, since $p\leq p_2$ and $p\neq 0$, one has $p(\bold{1}-f)\neq 0$. 
 It is well known that in a von Neumann algebra $\cN$, if $r,s \in \cN$ are nonzero projections, then $r\cN s\neq \{0\}$. Applying this to the von Neumann algebra $\cN=p\cM p$ with $r=p(\bold{1}-e)$ ans $s=p(\bold{1}-f)$, we get
 \[
 (p(\bold{1}-e)) (p\cM p) (p(\bold{1}-f))\neq \{0\}.
 \]

 Since $p$ is a central projection, it follows that $(p(\bold{1}-e)) (p\cM p) (p(\bold{1}-f))=p(\bold{1}-e)\cM (\bold{1}-f)p$. Therefore, $p(\bold{1}-e)\cM (\bold{1}-f)p\neq \{0\}$. Hence $(\bold{1}-e)\cM (\bold{1}-f) \neq \{0\}$, which is a contradiction of Equation \ref{eq:ext-cond}. Therefore, $p_2p_3=0$.

 Define $p_1=\bold{1}-p_2-p_3$. Then, $p_1,p_2,p_3\in \cM$  are central projections satisfying $p_1+p_2+p_3=\bold{1}$. Now, we show $p_1,p_2,p_3$ 
 satisfies the conditions (\ref{i1:u}), (\ref{i2:iso}) and (\ref{i3:coiso}). Because $\bold{1}-e\leq p_3$, we have $p_2(\bold{1}-e)\leq p_2p_3=0$. Thus $p_2(\bold{1}-e)=0$, that is, $p_2e=p_2$. Then $(p_2 x)^* (p_2 x)= p_2x^*x=p_2e=p_2$. Hence, $p_2x$ is an isometry in $p_2\cM$ unless $p_2=0$. It is non-unitary when $p_2\neq 0$; because $p_2(\bold{1}-f)\neq 0$ implies $p_2f \neq f$. But  $(p_2x)(p_2x)^*=p_2xx^*=p_2 f\neq f$. Therefore, it cannot be unitary in $p_2\mathcal {M}$. Similarly, one can show that if $p_3\neq 0$, then  $p_3x$ is a non-unitary coisometry in $p_3\cM$.  Now, we show that $p_1x$ is a unitary in $p_1\cM$, unless $p_1\neq 0$. Because $p_1$ and $p_3$ are mutually orthogonal projections and $(\bold{1}-e)\leq p_3$ implies $p_1(\bold{1}-e)\leq p_1p_3=0$. Thus $p_1(\bold{1}-e)=0$, that is, $p_1e=p_1$. Likewise, since $p_1$ and $p_2$ are mutually orthogonal projections and $(\bold{1}-f)\leq p_2$, we get $p_1f=p_1$. Therefore, $(p_1x)^*(p_1x)=p_1e=p_1$ and $(p_1x) (p_1x)^*=p_1f=p_1$. Thus, if $p_1\neq 0$, one has $p_1x$ is a unitary in $p_1\cM$. Consequently, $x$ is a $C^*$-extreme point of $\cM_1$.

 $(i) \implies(iii)$: Since $x$ is a $C^*$-extreme point of $\cM_1$, by Theorem \ref{t:A}, there exist central projections $p_1,p_2,p_3 \in \cM$ such that  conditions (\ref{i1:u}), (\ref{i2:iso}) and (\ref{i3:coiso}) hold. Now applying Corollary \ref{c:st.ext-vN} and Proposition \ref{p:st.ext-in-d.sum}, one has $x$ is a strongly extreme point of $\cM_1$. 

$(iii) \implies(ii)$: Follows immediately from the fact 
that
every strongly extreme point is a linear extreme point.
This completes the proof.
 \end{proof}

 \begin{remark}
    \label{rem:significance}
    Note that the classification scheme of Theorem~ \ref{t:C*ext-in-subclass} now identifies von Neumann algebras for which linear extreme points of the closed unit ball are precisely isometries and coisometries.
\end{remark}

 As a consequence of Theorem \ref{t:B}, we have the following results.

 \begin{corollary}
     \label{c:bidual}
     Let $x\in \cA_1$ be a linear extreme point of $\cA_1$. then $x$ is a $C^*$-extreme point of $\cA_1^{**}$.
 \end{corollary}

 \begin{proof}
     From Theorem \ref{t:ext-in-bidual}, we know that $x$ is a linear extreme point of $\cA_1^{**}$. Using the fact that $\cA^{**}$ is a von Neumann algebra and applying Theorem \ref{t:B}, we conclude that $x$ is a $C^*$-extreme point of $\cA_1^{**}$. This completes the proof.
 \end{proof}

  \begin{proposition}
       \label{c:vN}
Let $\cM$ be a von Neumann algebra.
Let the set of isolated points $S \subset \Omega$ be dense in $\Omega$ and $f$  be a $C^*$-extreme point of $C(\Omega,\cM)_1$. Then $f(\omega)$ be a $C^*$-extreme point of $\cM_1$ for all $\omega\in \Omega$.
 \end{proposition}

\begin{proof}
    Let $f$  be a $C^*$-extreme point of $C(\Omega,\cM)_1$. Then, by Corollary \ref{c:C_extonisopt}, we have $f(\omega)$ is a $C^*$-extreme point of $\cM_1$ for all $\omega\in S$.
 Now, take $\omega\in \Omega\setminus S$.  By assumption, there exists a net $\{\omega_\alpha\}\subset S$ of isolated points such that $\omega_\alpha \rightarrow \omega$.
  Since $f$ is continuous, it follows that $f(\omega_\alpha) \rightarrow f(\omega)$. 
   Applying Theorem \ref{t:B} and Remark \ref{rem:closed}, one has that the set of $C^*$-extreme points of $\cM_1$ is closed. 
  Because $f(\omega_\alpha)$ is a $C^*$-extreme point of $\cM_1$ for all $\alpha$,
  we conclude that
   $f(\omega)$ is a $C^*$-extreme point of $\cM_1$. This completes the proof.
\end{proof}

It was pointed out in \cite[Example~4.4]{BJR} that even for a finite-dimensional Banach space $E$, the pointwise geometric behaviour of a function $f \in C(\Omega, E)$ need not be reflected in the vector-valued case. In contrast, for the case of $C(\Omega, \mathcal M)$, where $\mathcal M$ is a von Neumann algebra, we have the following result.

\begin{proposition}
     \label{p:st.ext-in-cts}
     Let  $f\in C(\Omega, \cM)_1$ be such that $f(\omega)$ is a $C^*$-extreme point of $\cM_1$ for all $\omega \in \Omega$. Then $f$ is a strongly extreme point of $C(\Omega, \cM)_1$.
 \end{proposition}

 \begin{proof}
     Since $f(\omega)$ is a $C^*$-extreme point for all $\omega \in \Omega$, by Theorem \ref{t:B}, it follows that $f(\omega)$ is a strongly extreme point of $\cM_1$. 
     Therefore, $f$ is a linear extreme point of $C(\Omega,{\mathcal M})_1$. Now using the fact that $C(\Omega,{\mathcal M})$ is a $C^\ast$-algebra and Theorem \ref{t:B}, we conclude that $f$ is a strongly extreme point of $C(\Omega, \cM)_1$.
 \end{proof}

\vspace{.3cm}
 \textbf{Open Questions:}
 
 \textbf{Question~~1:} We do not know whether an analogue of Theorem~\ref{t:ideal} holds in the $C^*$-extreme case. 

 \textbf{Question~~2:} For a given $C^*$-algebra $\mathcal{A}$ and $f \in C(\Omega, \mathcal{A})_1$, does $f$ being $C^*$-extreme imply that $f(\omega)$ is $C^*$-extreme for all $\omega \in \Omega$?

 \textbf{Question~~3:} We do not know if $x$ is a $C^*$-extreme point of $\cA_1$ then whether it implies $x$ is also a $C^*$-extreme point of $\cA_1^{**}$.

\vspace{.3cm}
\textbf{Acknowledgements:} This work is part of the project ``Classification of Banach spaces using differentiability",
funded by the Anusandhan National Research Foundation (ANRF), Core Research Grant, CRG2023-000595.
The first author is a postdoctoral fellow in this project. The authors thank 
Mr. Chinmay Ajay Tamhankar (IITM) for his comments/suggestions on Theorem 4.8.

\end{document}